\documentclass[11pt]{article}
\usepackage[margin=1.25in]{geometry}
\usepackage{amsmath,amssymb,amsthm,mathtools}
\usepackage{microtype}
\usepackage{booktabs}
\usepackage{cite}
\usepackage{enumitem}
\usepackage[pagebackref]{hyperref}%
  \hypersetup{colorlinks,citecolor=blue,linkcolor=blue,urlcolor=blue}
\usepackage{xcolor}
\usepackage{tikz}
\usetikzlibrary{calc}
\expandafter\def\expandafter\normalsize\expandafter{%
  \normalsize
  \setlength\abovedisplayskip{7pt}%
  \setlength\belowdisplayskip{7pt}%
  \setlength\abovedisplayshortskip{0pt}%
  \setlength\belowdisplayshortskip{7pt}%
}

\newcommand{\M}{\mathring{M}}
\newcommand{\R}{\mathbb R}
\newcommand{\Q}{\mathbb Q}

\newcommand{\adj}{\operatorname{adj}}
\newcommand{\vol}{\operatorname{vol}}
\newcommand{\area}{\operatorname{area}}
\newcommand{\tr}{\operatorname{tr}}
\newcommand{\mb}[1]{\mathbf{#1}}
\newcommand{\K}{\mathsf{K}}
\newcommand{\E}{\mathsf{E}}
\newcommand{\F}{\mathsf{F}}
\newcommand{\B}{\mathsf{B}}
\newcommand{\T}{\mathsf{T}}
\newcommand{\SL}{\mathrm{SL}}
\newcommand{\SO}{\mathrm{SO}}

\newcommand{\ZArch}{\Z_{\mathrm p,3}}
\newcommand{\Z}{\mathcal{Z}}
\newcommand{\br}[1]{\langle{#1}\rangle}
\newcommand{\rms}{{\mathrm s}}
\newcommand{\rmh}{{\mathrm h}}

\theoremstyle{plain}
\newtheorem{theorem}{Theorem}[section]
\newtheorem{proposition}[theorem]{Proposition}
\newtheorem{lemma}[theorem]{Lemma}
\newtheorem{corollary}[theorem]{Corollary}

\theoremstyle{definition}

\title{The Truncated Octahedral Conjecture}
\author{Thomas Hales and Lark Song}
\date{}
\hypersetup{
    pdftitle={The Truncated Octahedral Conjecture},
    pdfauthor={Thomas Hales and Lark Song}
}

\begin{document}
\maketitle

\begin{abstract}
Among three-dimensional parallelohedra of fixed volume, the
Archimedean truncated octahedron uniquely minimizes surface area,
affirming a conjecture of K.~Bezdek from 2006.  One inequality in the
proof is verified by an exact computer-assisted certificate.
The simplex, frame, and graphical-zonotope inequalities used in the
proof hold in arbitrary dimension.
\end{abstract}

\section{Introduction}

A \emph{parallelohedron} is a convex polyhedron that admits a
face-to-face tiling of Euclidean three-space by translations.  In
1885, Fedorov proved that there are five combinatorially distinct
types: parallelepipeds, hexagonal prisms, rhombic dodecahedra,
elongated dodecahedra, and truncated octahedra \cite{Fedorov1885}.
Their structure was developed further by Minkowski and Venkov
\cite{Minkowski1897,Venkov1954}.  Venkov proved, and McMullen later
proved independently, that the face-to-face hypothesis can be dropped
\cite{Venkov1954,McMullen1980}.  A useful introduction is
\cite{Vallentin}.

For a convex body $\Z\subset\R^n$, $n\ge2$, write its isoperimetric quotient as
\[
 \iota_n(\Z):=\frac{\area(\partial\Z)}{\vol_n(\Z)^{(n-1)/n}},
\]
where $\area$ denotes $(n-1)$-dimensional boundary measure.

The Archimedean truncated octahedron is, up to scale, the Voronoi cell
of the body-centered cubic lattice.  Fix a representative $\ZArch$ of
this similarity class.  The subscript $\mathrm p$ refers to the
symmetric permutohedron, whose $n$-dimensional version
$\Z_{\mathrm p,n}$ is defined in Section~\ref{sec:graphical}.
K. Bezdek conjectured in 2006 that it
minimizes the isoperimetric quotient among parallelohedra
\cite[Conjecture~7.5]{Bezdek2006}; see also
\cite[Conjecture~1]{Langi2022}.  L\'angi proved the analogous result
for mean width \cite{Langi2022}.  Jo\'os and L\'angi obtained the
sharp surface area bound for rhombic dodecahedra \cite{JoosLangi2023}.

Cesaroni and Novaga recently proved that the body-centered cubic
lattice is a strict local minimizer among three-dimensional lattice
Voronoi cells \cite{CesaroniNovaga2026}.  Voronoi conjectured that
every parallelohedron is affinely equivalent to a lattice Voronoi cell
\cite{Voronoi1908}, which is known to hold through dimension five
\cite{Garber2025}.  Since the isoperimetric quotient is not affine
invariant, local minimality does not reduce to lattice Voronoi cells.
Song extended the local minimality from the lattice-Voronoi subclass
to the full parameter space of parallelohedra \cite{Song2026}. Shortly
after our preprint was posted to the arXiv, Cesaroni and Novaga
announced an independent proof of the main result,
Theorem~\ref{thm:global-main}~\cite{CesaroniNovaga-sep2026}.

The Kelvin problem asks for a least-area partition of space into cells of equal
volume.  Kelvin conjectured that a perturbation of the Archimedean
truncated-octahedral tiling is optimal; Weaire and Phelan later found a
counterexample \cite{Kelvin,Weaire}.  Various
revised conjectures add assumptions on the structure of the foam that
make the Kelvin cell a plausible candidate for the optimal
partition~\cite{Sullivan}.  Bezdek's conjecture can be interpreted as
a modification of the original Kelvin problem that restricts the 
equal-volume partitions of space to tilings by parallelohedra.

Although Bezdek's conjecture provides the clearest statement, the
conjecture has historical precedents.  Because of the way it provides
the scaffolding for the Kelvin cell, the Archimedean truncated
octahedron has made recurrent appearances in isoperimetric
investigations.  Lewis supplied polyhedral calculations and
comparisons~\cite{Lewis1925FurtherStudy}; Matzke can arguably be
interpreted as asserting the optimality of the Archimedean truncated
octahedron among space-filling convex
polyhedra~\cite[p.~341]{Matzke1927Analysis}; in 1943 Sharp stated,
under the additional assumption of equal-edge lengths, that a uniform
equal-volume polyhedral division of minimum surface area has cells
that are Archimedean truncated octahedra~\cite{Sharp1943Fundamentals}.
D'Arcy Thompson likewise described the Archimedean truncated
octahedron as the homogeneous space-filler enclosing a given volume
with least surface area \cite[p.~734]{Thompson1942GrowthForm}. In
Kelvin's terminology, following Bravais, a homogeneous division has
congruent translational cells and hence, in the convex polyhedral
case, consists of parallelohedra~\cite{Kelvin1894Homogeneous}.

This article gives a proof of Bezdek's conjecture in partnership with
AI:

\begin{theorem}[truncated octahedral conjecture]\label{thm:global-main}
    Every parallelohedron $\Z\subset\R^3$ satisfies
\begin{equation}\label{eq:global-main}
    \iota_3(\Z)\geq\iota_3(\ZArch) = \frac{3(1+2\sqrt3)}{4^{2/3}}.
\end{equation}
Equality holds if and only if $\Z$ is an Archimedean truncated octahedron.
\end{theorem}

\subsection*{Outline of the proof}

Fedorov's classification reduces the main theorem to five
combinatorial types.  The parallelepiped and hexagonal-prism cases
follow from the classical sharp inequality for polygonal prisms.
The remaining arguments have a common higher-dimensional formulation
in terms of cofactors and spanning trees.

Section~\ref{sec:simplex} proves a weighted
simplex inequality and its reciprocal balanced-frame form.  Applied
to the adjugate of a zone matrix, the latter gives the sharp bound
for the $(n+1)$-generator rhombic class and an elongated family with
one additional subset-sum generator.  Their three-dimensional
specializations exceed the Archimedean truncated-octahedral value.

For the weighted $\K_{n+1}$-graphical family, the zone matrix
$B_\beta$ has determinant equal to the volume.  Its adjugate
records the facet coefficients, indexed by cuts of the graph.
The canonical affine map $d^{1/(2n)}B_\beta^{-1/2}$, where
$d=\det B_\beta$, determines a surface-area tensor $H_\beta$.
A trace estimate removes the arbitrary special-linear deformation,
and one calibrated Jensen inequality gives
\[
 \area(\partial\Z(\beta,M))
 \ge2n(\det H_\beta)^{1/n}
 \ge2n d^{(n-1)/n}
               \left(\frac{\varphi^3}{\psi}\right)^{1/(2n)}.
\]
Here $\varphi$ and $\psi$ are explicit polynomial scalars.
Equality in Jensen forces all normalized quadratic values on the
cut normals to agree, making $B_\beta$ scalar and all zone weights
equal.  The affine equality condition then forces $M$ to be orthogonal.

Section~\ref{sec:three-dimensional} specializes these results to
the five Fedorov types.  For the truncated-octahedral type the one
remaining input is $f_3=\varphi^3-\kappa_3^2\psi\ge0$.
This homogeneous degree-$18$ inequality in six variables is proved
in the appendix by exact Bernstein coefficients in $\Q(\sqrt3)$
on a finite cover.  Nonnegativity alone suffices, since geometric
uniqueness has already been established by the analytic argument.

\subsection*{Notation}

Bold lowercase Roman letters denote vectors in $\R^n$; the main
letters are $\mb u,\mb v,\mb w,\mb n$ for balanced frames, simplex
vertices, reciprocal edges, and normals.  We use $\mb c_I$ for a
cofactor vector indexed by a set of generating vectors, reserving
$\mb w_{ij}$ for a reciprocal edge indexed by its two endpoints.
Uppercase Roman letters denote matrices, except that $I,J,S$ are
also used for finite index sets.  We write $I_m$ for the $m\times m$
identity matrix, $\mb e_i$ for the $i$th standard basis vector, and
$\mb1$ for the column vector all of whose entries are one; the
ambient dimension is specified by the context.  Thus $\mb1^T$ is
the corresponding row vector.  The Kronecker delta $\delta_{ij}$ is
one when $i=j$ and zero otherwise; it is distinct from the positive
determinant scalar $\delta_{\mb u}$ used for a balanced frame.
We write $\adj M$ for the adjugate.  The symbol $\Z$ denotes
a zonotope or polyhedron.  Uppercase sans serif letters denote
graphs and combinatorial collections.  Discrete indices include
$i,j,k,\ell,b,e$, and the dimension is $n$.  Lowercase Roman and
Greek letters denote real scalars and functions.  Superscripts $T$
and $c$ denote transpose and set complement.  Hats indicate the
specified rescaling, not necessarily unit length; a circle indicates
canonical affine position.  We use $|\cdot|$ for absolute value,
Euclidean norm, and cardinality.  A \emph{zone} is a direction of
generating segments.  In dimension three, subscripts \emph{h} and
\emph{s} refer to hexagonal and square facets.

\section{Cofactors, simplices, and balanced frames}
\label{sec:simplex}

Throughout Sections~\ref{sec:simplex}--\ref{sec:averaging}, the dimension
is $n\ge2$.  Let $\K_{n+1}$ be the complete graph on vertex set $[n+1]=\{1,\ldots,n+1\}$,
with edge set $\E=\binom{[n+1]}2$, and spanning trees $\T$.  For edge
weights $\rho_e\ge0$, write
\begin{equation}\label{eq:tree-polynomial}
 \tau(\rho):=\sum_{J\in\T}\prod_{e\in J}\rho_e,
 \qquad c_n:=\frac{n}{\sqrt2}(n+1)^{1/(2n)}.
\end{equation}

\subsection{Determinant and zonotope formulas}

For vectors $\mb v_1,\ldots,\mb v_m\in\R^n$ and
$I=\{i_1<\cdots<i_{n-1}\}\subset[m]$, define the cofactor vector
$\mb c_I$ by
\begin{equation}\label{eq:cofactor-definition}
 \mb c_I\cdot\mb x
 =\det(\mb v_{i_1},\ldots,\mb v_{i_{n-1}},\mb x)
 \qquad(\mb x\in\R^n).
\end{equation}
Its coordinates are signed minors.  In dimension three it is
$\mb v_{i_1}\times\mb v_{i_2}$.  Reordering $I$ changes only its sign.
If $a_i\ge0$ and $Y$ has columns $\sqrt{a_i}\mb v_i$, Cauchy--Binet gives
\begin{equation}\label{eqn:cauchy-binet}
 \det(YY^T)
 =\sum_{|I|=n}\left(\prod_{i\in I}a_i\right)
                  \det(\mb v_i:i\in I)^2.
\end{equation}
The corresponding identity for the adjugate is equally direct.

\begin{lemma}\label{lemma:adj}
With the notation above,
\begin{equation}\label{eq:adj}
 \adj(YY^T)=\sum_{|J|=n-1}
             \left(\prod_{j\in J}a_j\right)\mb c_J\mb c_J^T.
\end{equation}
\end{lemma}

\begin{proof}
Absorb the constants $\sqrt{a_i}$ into $\mb v_i$ and $\mb c_I$, and
assume without loss of generality that $a_i=1$ for all $i$.  The right
side of \eqref{eq:adj} then has the form $CC^T$, where $C$ has columns
$\mb c_I$.  Abbreviate $d_I=\det(\mb{v}_i:i\in I)$, and let $Y_I$ be
the submatrix of $Y$ with columns $I$.  We have for appropriate
signs $\sigma_{j,J}=\pm1$, using the cofactor formula for the adjugate
of $Y_I$:
\begin{align*}
 YY^TCC^T &= \sum_{j,J} \mb v_j (\mb v_j\cdot \mb c_J)\mb c_J^T
 =\sum_{|I|=n} d_I\sum_{J\cup \{j\}=I}\sigma_{j,J} \mb v_j\mb c_J^T\\
 &=\sum_{|I|=n}d_IY_I\adj Y_I
 =\sum_{|I|=n}d_I^2I_n
 =\det(YY^T)I_n.
\end{align*}
The last equality is Cauchy--Binet.
If $Y$ has rank $n$, this identity proves
\eqref{eq:adj}.  For the remaining cases, first pad $Y$ with zero
columns if $m<n$.  Both sides are polynomial in the entries of $Y$,
and full-rank matrices are dense, so the same identity holds without
the rank assumption.
\end{proof}

For the full-dimensional zonotope
$\Z=\sum_{i=1}^m[-a_i\mb v_i/2,a_i\mb v_i/2]$, the standard
volume and surface-area formulas are
\begin{align}
 \vol_n(\Z)
 &=\sum_{|I|=n}\left(\prod_{i\in I}a_i\right)
                    |\det(\mb v_i:i\in I)|,
                    \label{eq:zonotope-volume}\\
 \frac12\area(\partial\Z)
 &=\sum_{|I|=n-1}\left(\prod_{k\in I}a_k\right)|\mb c_I|.
                    \label{eq:zonotope-area}
\end{align}
See \cite{Shephard1974} and \cite[Corollary~1]{JoosLangi2023}.
Thus the same cofactors occur in the adjugate and in surface area.
Under an invertible linear map $M$, they transform by
$\mb c_I\mapsto\det(M)M^{-T}\mb c_I$.

\subsection{A weighted simplex inequality}

\begin{theorem}[weighted simplex inequality]\label{thm:weighted-simplex}
Let $\mb v_1,\ldots,\mb v_{n+1}\in\R^n$ be affinely independent, and put
\[
 \ell_{ij}:=|\mb v_i-\mb v_j|,
 \qquad
 \delta:=|\det(\mb v_2-\mb v_1,\ldots,\mb v_{n+1}-\mb v_1)|.
\]
For nonnegative edge weights $\rho_e$ with $\tau(\rho)>0$,
\begin{equation}\label{eq:weighted-simplex}
 \sum_{e\in\E}\rho_e\ell_e
 \ge c_n\bigl(\delta\,\tau(\rho)\bigr)^{1/n}.
\end{equation}
Equality holds if and only if the simplex is regular and all the
edge weights are equal.
\end{theorem}

\begin{proof}
Put $s:=\sum_e\rho_e\ell_e$, and set
\[
 \hat\rho_e:=\frac{\rho_e\ell_e}{s},\qquad
 \widehat{\mb v}_{ij}:=\frac{\mb v_i-\mb v_j}{\ell_{ij}},\qquad
 G:=\sum_e\hat\rho_e\widehat{\mb v}_e\widehat{\mb v}_e^T.
\]
The hypothesis $\tau(\rho)>0$ means that the positive-weight edges
contain a spanning tree.  In particular $s>0$,
$\sum_e\hat\rho_e=1$, and the vectors with positive coefficients
in $G$ span $\R^n$.  Hence $G$ is positive definite, and
$\tr G=\sum_e\hat\rho_e|\widehat{\mb v}_e|^2=1$.

A set of $n$ simplex edges is independent precisely when it is a
spanning tree.  A cycle gives a linear dependence; an acyclic graph
on $n+1$ vertices with $n$ edges is a tree.  Conversely, successively
deleting leaves, or performing the corresponding elementary column
operations, shows that the determinant of a tree's edge vectors has
the same absolute value $\delta$ as that of the star at vertex $1$.
Cauchy--Binet therefore expands the determinant over exactly these
trees.  Applying trace AM--GM next gives
\begin{equation}\label{eq:simplex-geometric-bound}
 \sum_{J\in\T}\left(\prod_{e\in J}\hat\rho_e\right)
       \left(\frac{\delta}{\prod_{e\in J}\ell_e}\right)^2
 =\det G\\
 \le\left(\frac{\tr G}{n}\right)^n=n^{-n}.
\end{equation}
Equality in this determinant bound holds precisely when $G=I_n/n$.

Let $L$ be the weighted Laplacian on $[n+1]$ with weights
$\hat\rho_e$:
\[
 L_{ij}:=\begin{cases}
 -\hat\rho_{ij},&i\ne j,\\
 \displaystyle\sum_{k\ne i}\hat\rho_{ik},&i=j.
 \end{cases}
\]
Each row sums to zero, so $L\mb1=0$.  Moreover, for
$\mb z=(z_1,\ldots,z_{n+1})^T$,
\[
 \mb z^TL\mb z
 =\sum_{i<j}\hat\rho_{ij}(z_i-z_j)^2.
\]
Since the positive-weight edges connect all vertices, this expression
vanishes exactly when $\mb z$ is constant.  Thus $\ker L=\R\mb1$:
one eigenvalue is zero and the remaining $n$ eigenvalues
$\lambda_1,\ldots,\lambda_n$ are positive.  Their sum is
$\tr L=2\sum_e\hat\rho_e=2$.  The eigenvalue form of the
weighted matrix-tree theorem gives
$\tau(\hat\rho)=(n+1)^{-1}\lambda_1\cdots\lambda_n$
\cite{klee2019linear,DuvalKlivansMartin2009}.  Scalar AM--GM yields
\begin{equation}\label{eq:simplex-tree-bound}
 \tau(\hat\rho)
 =\frac1{n+1}\prod_{i=1}^n\lambda_i
 \le\frac1{n+1}
       \left(\frac{\lambda_1+\cdots+\lambda_n}{n}\right)^n
 =\frac1{n+1}\left(\frac2n\right)^n.
\end{equation}
Cauchy--Schwarz, followed by the two determinant bounds, now gives
\begin{align*}
 \frac{\delta\,\tau(\rho)}{s^n}
 =\sum_{J\in\T}\left(\prod_{e\in J}\hat\rho_e\right)
                 \frac{\delta}{\prod_{e\in J}\ell_e}
 \le\sqrt{\tau(\hat\rho)\det G}
 \le\frac{2^{n/2}}{\sqrt{n+1}\,n^n}.
\end{align*}
Rearranging proves \eqref{eq:weighted-simplex}.

Suppose equality holds.  Equality in
\eqref{eq:simplex-tree-bound} forces equal eigenvalues $2/n$.
Since $L$ is symmetric, it acts
as zero on $\R\mb1$ and as multiplication by $2/n$ on its orthogonal
complement $\mb1^\perp$.  The orthogonal projection onto that
complement is $I_{n+1}-(n+1)^{-1}\mb1\mb1^T$.  Therefore
\[
 L=\frac2n\left(I_{n+1}-\frac1{n+1}\mb1\mb1^T\right),
 \qquad \hat\rho_e=\frac2{n(n+1)}\quad(e\in\E),
\]
where the second identity follows by comparing off-diagonal
entries.  In particular, every tree has positive weight.
Equality in Cauchy--Schwarz then says that
$\delta/\prod_{e\in J}\ell_e$, and hence
$\prod_{e\in J}\ell_e$, is independent of $J\in\T$.
For distinct $i,j,k$, compare the tree with edges $ij,jk$ and every
remaining vertex attached to $k$ with the tree obtained by replacing
$ij$ by $ik$.  Cancelling the common positive factors gives
$\ell_{ij}=\ell_{ik}$.  Varying the three vertices shows that all
edge lengths are equal.  Since $\hat\rho_e=\rho_e\ell_e/s$,
all the original weights $\rho_e$ are equal as well.

Conversely, suppose the simplex is regular, with common edge length
$\ell$, and the weights are equal.  To verify the geometric equality
explicitly, translate the centroid to zero and let
$Q=(\mb q_1,\ldots,\mb q_{n+1})$ be the centered vertex matrix.
The equal edge lengths and $\sum_i\mb q_i=0$ give
\[
 Q^TQ=\frac{\ell^2}{2}
       \left(I_{n+1}-\frac1{n+1}\mb1\mb1^T\right),
 \qquad QQ^T=\frac{\ell^2}{2}I_n.
\]
The second identity follows because $Q$ has rank $n$ and $QQ^T$
has the same nonzero eigenvalues as $Q^TQ$.  Using the balance
identity once more,
\[
 \sum_{i<j}(\mb q_i-\mb q_j)(\mb q_i-\mb q_j)^T
 =(n+1)QQ^T=\frac{(n+1)\ell^2}{2}I_n.
\]
Thus $\hat\rho_e=2/(n(n+1))$ indeed gives $G=I_n/n$.
The Laplacian has the required equal positive eigenvalues, and all
tree edge-length products are equal.  Every inequality above is
therefore an equality.
\end{proof}

There is an exact isoperimetric interpretation.  The sum
$\sum_e\rho_e\ell_e$ is the first intrinsic volume of
$\Z_\rho=\sum_{i<j}[0,\rho_{ij}(\mb v_i-\mb v_j)]$, and
$\vol_n(\Z_\rho)=\delta\tau(\rho)$.
The first intrinsic volume is proportional to mean width
\cite[Corollary~1 and Remark~1]{JoosLangi2023}.
Thus the inequality is a sharp mean-width--volume inequality for
weighted simplex-edge zonotopes.  In dimension three it also follows
from L\'angi's theorem \cite{Langi2022}.

\subsection{Balanced frames and reciprocal edges}

A \emph{balanced $(n+1)$-frame} is a tuple
$(\mb u_i)_{i=1}^{n+1}$ in $\R^n$ whose sum is zero and for which any
$n$ vectors are independent.  Its maximal determinants have a common
absolute value
$\delta_{\mb u}:=|\det(\mb u_1,\ldots,\mb u_n)|>0$.
The frame is \emph{regular} if, for some $a>0$,
\begin{equation}\label{eq:regular-frame}
 \mb u_i\cdot\mb u_j=a\bigl((n+1)\delta_{ij}-1\bigr).
\end{equation}
For $i\ne j$, define the reciprocal edge vector $\mb w_{ij}$ by
\begin{equation}\label{eq:reciprocal-pairing}
 \mb u_k\cdot\mb w_{ij}
 =\delta_{\mb u}(\delta_{ki}-\delta_{kj})
 \qquad(k\in[n+1]).
\end{equation}
The balance relation makes these equations consistent: the sum of
the right sides is zero, just as the sum of the left sides is zero.
Any $n$ of the equations determine $\mb w_{ij}$ uniquely, and the
remaining equation follows from their sum.  In particular,
$\mb w_{ji}=-\mb w_{ij}$.
In terms of the earlier cofactor convention, 
$\mb w_{ij}=\pm\mb c_{[n+1]\setminus\{i,j\}}$.
The sign is fixed by \eqref{eq:reciprocal-pairing}, but
it will be
immaterial in formulas below.

\begin{corollary}[balanced-frame inequality]\label{cor:balanced-frame}
For a balanced frame and nonnegative weights $\rho_{ij}$, put
\[
 A:=\sum_{i<j}\rho_{ij}\mb w_{ij}\mb w_{ij}^T.
\]
If $A$ is positive definite, then
\begin{equation}\label{eq:balanced-frame}
 \sum_{i<j}\rho_{ij}|\mb w_{ij}|
 \ge c_n\left(\frac{\det A}{\delta_{\mb u}^{\,n-1}}\right)^{1/n}.
\end{equation}
Equality holds if and only if the frame is regular and all
$\rho_{ij}$ are equal.
\end{corollary}

\begin{proof}
Write $U=(\mb u_1,\ldots,\mb u_n)$ and choose simplex vertices
\[
 \mb v_{n+1}=0,\qquad
 \mb v_i=\delta_{\mb u}U^{-T}\mb e_i\quad(1\le i\le n).
\]
For $k\le n$, pairing with $\mb u_k$ gives
$\mb u_k\cdot\mb v_i=\delta_{\mb u}\delta_{ki}$.
The equation for $\mb u_{n+1}=-\sum_{k=1}^n\mb u_k$ then follows
by balance.  Thus the differences satisfy
\eqref{eq:reciprocal-pairing}, and uniqueness gives
$\mb w_{ij}=\mb v_i-\mb v_j$.  Their simplex determinant is
\[
 \delta_{\mb v}
 =\left|\det(\delta_{\mb u}U^{-T})\right|
 =\frac{\delta_{\mb u}^{\,n}}{|\det U|}
 =\delta_{\mb u}^{\,n-1}.
\]
As in the preceding proof, every tree of reciprocal edges has this
absolute determinant.  Cauchy--Binet therefore gives
$\det A=\delta_{\mb v}^{\,2}\tau(\rho)$.
Since $A$ is positive definite, $\tau(\rho)>0$, and
Theorem~\ref{thm:weighted-simplex} applies.  Substituting
$\delta_{\mb v}\tau(\rho)=\det A/\delta_{\mb v}$ gives
\eqref{eq:balanced-frame}.

For equality, the reciprocal simplex is regular precisely when its
edge-column matrix $V=(\mb v_1,\ldots,\mb v_n)$ satisfies
$V^TV=b(I_n+\mb1\mb1^T)$ for some $b>0$.  Since
\[
 V^TV=\delta_{\mb u}^{\,2}(U^TU)^{-1},\qquad
 (I_n+\mb1\mb1^T)^{-1}
 =I_n-\frac1{n+1}\mb1\mb1^T,
\]
the regular-simplex condition is equivalent to
\[
 U^TU=\frac{\delta_{\mb u}^{\,2}}{b}
             \left(I_n-\frac1{n+1}\mb1\mb1^T\right).
\]
Its entries give \eqref{eq:regular-frame} for $i,j\le n$, with
$a=\delta_{\mb u}^{\,2}/(b(n+1))$.  The pairings involving
$\mb u_{n+1}$ follow from $\mb u_{n+1}=-\sum_{i=1}^n\mb u_i$.
The converse follows by reversing the calculation.  Hence the
frame is regular exactly when the reciprocal simplex is regular,
and the equality statement follows from
Theorem~\ref{thm:weighted-simplex}.
\end{proof}

\section{Rhombic and elongated zonotopes}
\label{sec:rhombic}

The balanced-frame inequality applies directly to the adjugate of a
zone matrix.  It gives the sharp surface-area bound for the
$(n+1)$-generator rhombic class and for an elongated extension with
one additional subset-sum direction.

\subsection{The rhombic class}

Let $\Z_{\mathrm r,n}$ denote a zonotope generated by the vectors of a
regular balanced $(n+1)$-frame, with equal weights.  Its similarity
class does not depend on the frame or the common weight.  This is the
projection of an $(n+1)$-cube along a main diagonal, up to similarity;
for $n=3$ it is Kepler's regular rhombic dodecahedron.

\begin{theorem}[rhombic zonotopes]\label{thm:rhombic}
Let $(\mb u_i)_{i=1}^{n+1}$ be a balanced frame and $\beta_i>0$.
Then
\[
 \Z:=\sum_{i=1}^{n+1}[-\beta_i\mb u_i/2,\beta_i\mb u_i/2]
\]
satisfies
\begin{equation}\label{eq:rhombic}
 \iota_n(\Z)\ge 2c_n
 =n\sqrt2\,(n+1)^{1/(2n)}.
\end{equation}
Equality holds if and only if the frame is regular and all $\beta_i$
are equal; equivalently, $\Z$ is similar to $\Z_{\mathrm r,n}$.
\end{theorem}

This is the surface-area case of a stronger intrinsic-volume theorem
of Jo\'os and L\'angi \cite[Theorem~3]{JoosLangi2023}.  We give a
tree-polynomial proof, since the same argument will also treat the
elongated family.

\begin{proof}
Set $B_\beta:=\sum_i\beta_i\mb u_i\mb u_i^T$ and $d:=\det B_\beta$.
All maximal minors of the frame have absolute value $\delta_{\mb u}$.
By \eqref{eqn:cauchy-binet} and \eqref{eq:zonotope-volume},
\begin{equation}\label{eq:rhombic-det-volume}
 d=\delta_{\mb u}\vol_n(\Z).
\end{equation}
With $\mb w_{ij}$ as in \eqref{eq:reciprocal-pairing}, put
$\gamma_{ij}:=\prod_{k\notin\{i,j\}}\beta_k$.
The adjugate and surface-area formulas give
\[
 \adj B_\beta=\sum_{i<j}\gamma_{ij}\mb w_{ij}\mb w_{ij}^T,
 \qquad
 \frac12\area(\partial\Z)=\sum_{i<j}\gamma_{ij}|\mb w_{ij}|.
\]
Applying Corollary~\ref{cor:balanced-frame} to $\adj B_\beta$ yields
\begin{equation}\label{eq:rhombic-chain}
 \frac12\area(\partial\Z)
 \ge c_n\left(\frac{\det(\adj B_\beta)}{
                          \delta_{\mb u}^{\,n-1}}\right)^{1/n}
 =c_n\vol_n(\Z)^{(n-1)/n}.
\end{equation}
Equality forces a regular frame and equal $\gamma_{ij}$.
Comparing two products whose omitted pairs share an index shows
that all $\beta_i$ are equal.  The converse follows from the equality
case of Corollary~\ref{cor:balanced-frame}.
\end{proof}

Every configuration of $n+1$ generators with every $n$ independent
can be written in this form: orient the generators so that their
unique dependence has positive coefficients, and absorb these
coefficients into a balanced frame.  Reversing a centered generating
segment does not change the zonotope.

\subsection{One additional subset-sum direction}

\begin{theorem}[elongated zonotopes]\label{thm:elongated}
Let $n\ge3$, let $(\mb u_i)_{i=1}^{n+1}$ be a balanced frame, and let
$S\subset[n+1]$ satisfy $|S|,|S^c|\ge2$.  Put
\[
 \mb u_S:=\sum_{i\in S}\mb u_i=-\sum_{j\in S^c}\mb u_j.
\]
For $\beta_i>0$ and $t\ge0$, the zonotope
\begin{equation}\label{eq:elongated-family}
 \Z_t:=\sum_{i=1}^{n+1}[-\beta_i\mb u_i/2,\beta_i\mb u_i/2]
                +[-t\mb u_S/2,t\mb u_S/2]
\end{equation}
satisfies $\iota_n(\Z_t)\ge2c_n$.  Equality holds if and only if
$t=0$, the frame is regular, and all $\beta_i$ are equal.
For $t>0$ the bound is strict and is sharp on the closure of this family.
\end{theorem}

\begin{proof}
Set $B_\beta:=\sum_i\beta_i\mb u_i\mb u_i^T+t\mb u_S\mb u_S^T$.
Every nonzero maximal minor of the unweighted configuration
$(\mb u_1,\ldots,\mb u_{n+1},\mb u_S)$ has absolute value
$\delta_{\mb u}$.  To see this for a minor involving $\mb u_S$,
let $i,j$ be the two omitted core indices and put
$W_{ij}:=\operatorname{span}\{\mb u_k:k\notin\{i,j\}\}$.
The balance relation says that $\mb u_j$ is congruent to
$-\mb u_i$ modulo the other $n-1$ core vectors:
\[
 \mb u_j\equiv-\mb u_i\pmod{W_{ij}},\qquad
 \mb u_S\equiv
 \begin{cases}
  0,&i,j\text{ on the same shore},\\
  \mb u_i,&i\in S,\ j\in S^c,\\
  -\mb u_i,&i\in S^c,\ j\in S
 \end{cases}
 \pmod{W_{ij}}.
\]
Adding a vector of $W_{ij}$ to the remaining column does not change
the determinant.  The minor is therefore zero in the first case and
has absolute value $\delta_{\mb u}$ in the other two cases.
Cauchy--Binet \eqref{eqn:cauchy-binet} sums the squares of these
minors, whereas the volume formula \eqref{eq:zonotope-volume} sums
their absolute values, with the same weight products.  Consequently
\begin{equation}\label{eq:elongated-det-volume}
 \det B_\beta=\delta_{\mb u}\vol_n(\Z_t).
\end{equation}

Nor does the extra zone introduce a new cofactor direction.
A cofactor involving $\mb u_S$ contains the core vectors with indices
in some set $K$ of size $n-2$.  Put
$W_K:=\operatorname{span}\{\mb u_k:k\in K\}$ and let
$\{i,j,\ell\}=[n+1]\setminus K$.  Then
\[
 \mb u_i+\mb u_j+\mb u_\ell\equiv0\pmod{W_K},\qquad
 \mb u_S\equiv\sum_{k\in S\setminus K}\mb u_k\pmod{W_K}.
\]
If the three omitted indices are all on the same shore,
$\mb u_S$ is congruent to zero modulo $W_K$.  Otherwise it is
congruent to the uniquely omitted vector on one shore, or to the
negative of that vector, modulo $W_K$.  Multilinearity of the
cofactor now shows that every nonzero new cofactor is
$\pm\mb w_{ij}$ for a core pair $ij$.  Grouping equal cofactor lines gives
\begin{equation}\label{eq:elongated-adj-area}
 \adj B_\beta=\sum_{i<j}\gamma_{ij}\mb w_{ij}\mb w_{ij}^T,
 \qquad
 \frac12\area(\partial\Z_t)=\sum_{i<j}\gamma_{ij}|\mb w_{ij}|,
\end{equation}
where
\begin{equation}\label{eq:elongated-coefficients}
 \gamma_{ij}=\left(\prod_{k\notin\{i,j\}}\beta_k\right)
 \begin{cases}
  1+t\displaystyle\sum_{k\in S^c}\beta_k^{-1},&i,j\in S,\\[4pt]
  1+t\displaystyle\sum_{k\in S}\beta_k^{-1},&i,j\in S^c,\\[4pt]
  1,&i,j\text{ on different shores}.
 \end{cases}
\end{equation}
For example, if $i,j\in S$, the vector $\mb u_S$ is a sum of the
$S^c$-generators up to sign, and in a nonzero cofactor it can replace
exactly one of them.  This gives the first line; the other cases
follow in the same way.  Applying Corollary~\ref{cor:balanced-frame}
to \eqref{eq:elongated-adj-area} and using
\eqref{eq:elongated-det-volume} gives the chain
\eqref{eq:rhombic-chain} with $\Z_t$ in place of $\Z$.

For equality, the frame is regular and every $\gamma_{ij}$ is the
same.  Comparing cross-shore pairs first gives
$\beta_i=a$ on $S$ and $\beta_j=b$ on $S^c$.
Comparing within-shore pairs with cross-shore pairs then gives
\[
 a=b+|S^c|t,\qquad b=a+|S|t.
\]
Thus $(n+1)t=0$ and $a=b$.  These conditions are sufficient by
Theorem~\ref{thm:rhombic}.  Taking equal core weights in a regular
frame and letting $t\downarrow0$ proves sharpness.
\end{proof}

For $n=3$ both shores have two elements, giving the usual elongated
dodecahedral direction.  The theorem concerns the subset-sum family
\eqref{eq:elongated-family}, not arbitrary $(n+2)$-generator zonotopes.

\section{Graphical zonotopes and canonical position}
\label{sec:graphical}

\subsection{The regular frame and the zone matrix}

Fix a regular balanced frame $\mb n_1,\ldots,\mb n_{n+1}\in\R^n$
normalized by $|\det(\mb n_1,\ldots,\mb n_n)|=1$.  Its Gram matrix is
\begin{equation}\label{eq:normal-gram}
 \mb n_i\cdot\mb n_j
 =(n+1)^{-(n-1)/n}\bigl((n+1)\delta_{ij}-1\bigr).
\end{equation}
Its reciprocal edge vectors, as defined in
\eqref{eq:reciprocal-pairing}, are characterized by
\begin{equation}\label{eq:zone-vector}
 \mb n_k\cdot\mb w_{ij}=\delta_{ki}-\delta_{kj},
 \qquad
 \mb w_{ij}=(n+1)^{-1/n}(\mb n_i-\mb n_j).
\end{equation}
For $\beta=(\beta_e)_{e\in\E}>0$ and $M\in\SL(n,\R)$, consider
\begin{equation}\label{eq:param}
 \Z(\beta,M):=
 \sum_{ij\in\E}
 [-\beta_{ij}M\mb w_{ij}/2,\beta_{ij}M\mb w_{ij}/2].
\end{equation}
This is a weighted $\K_{n+1}$-graphical zonotope with a
volume-preserving affine deformation.  Write
$\Z_{\mathrm p,n}:=\Z(\mb1,I_n)$ for the symmetric permutohedron;
for $n=3$ this is the Archimedean truncated octahedron.
Define
\begin{equation}\label{eq:Bbeta}
 B_\beta:=\sum_{e\in\E}\beta_e\mb w_e\mb w_e^T,
 \qquad d=d(\beta):=\det B_\beta>0.
\end{equation}

The facet pairs are indexed by the nontrivial unordered
bipartitions $b=\br{S\mid S^c}$ of $[n+1]$; denote this set by $\F$.
For each $b$, choose either shore and put
\begin{equation}\label{eq:bond-normal}
 \mb n_b:=\sum_{i\in S}\mb n_i,
 \qquad
 \gamma_b:=\tau_S(\beta)\tau_{S^c}(\beta).
\end{equation}
Here $\tau_S$ is the spanning-tree polynomial of the complete graph
induced on $S$, and $\tau_{\{i\}}=1$.  Complementing $S$ changes only
the sign of $\mb n_b$, which will not affect any formula.
The corresponding \emph{cut} is the set of edges with one endpoint
in each shore; we call $\mb n_b$ a \emph{cut normal}.
A \emph{singleton cut} has one shore $\{i\}$; we abbreviate it by
$\br i$ and may choose its normal to be $\mb n_i$.
There are $2^n-1$ facet-normal lines.  Moreover,
\begin{equation}\label{eq:normal-length}
 |\mb n_b|^2=(n+1)^{-(n-1)/n}|S|\,|S^c|.
\end{equation}

\begin{proposition}[tree and cut formulas]\label{prop:tree-cut}
For the zonotope \eqref{eq:param},
\begin{align}
 \vol_n(\Z(\beta,M))&=d=\tau(\beta),\label{eq:volume}\\
 \adj B_\beta&=\sum_{b\in\F}\gamma_b\mb n_b\mb n_b^T,
                                      \label{eq:adjB}\\
 \frac12\area(\partial\Z(\beta,M))
 &=\sum_{b\in\F}\gamma_b|M^{-T}\mb n_b|.
                                      \label{eq:surface-area}
\end{align}
The area of one facet in pair $b$ is $\gamma_b|M^{-T}\mb n_b|$.
Finally, the quadratic value of $B_\beta$ on a cut normal is the cut weight:
\begin{equation}\label{eq:rb}
 r_b:=\mb n_b^TB_\beta\mb n_b
 =\sum_{\substack{i\in S\,,\ j\in S^c}}\beta_{ij}.
\end{equation}
\end{proposition}

\begin{proof}
Put $U=(\mb n_1,\ldots,\mb n_n)$.  By \eqref{eq:zone-vector},
$\mb w_{i,n+1}=U^{-T}\mb e_i$ and
$\mb w_{ij}=U^{-T}(\mb e_i-\mb e_j)$ for $i,j\le n$.
Thus a set of $n$ generators is independent exactly when it is a
spanning tree.  Its determinant has absolute value one, since the
reduced incidence matrix of a tree has determinant $\pm1$ and
$|\det U|=1$.  Equations \eqref{eqn:cauchy-binet} and
\eqref{eq:zonotope-volume} prove \eqref{eq:volume}.

An independent set of $n-1$ generators is a forest with two
components, say $S$ and $S^c$.  Its cofactor is parallel to $\mb n_b$:
both are orthogonal to every internal edge of these components.
Adding an edge across the cut makes a spanning tree, and its pairing
with $\mb n_b$ is $\pm1$.  The cofactor is therefore exactly
$\pm\mb n_b$.  The weighted sum of the forests with these components
is $\tau_S(\beta)\tau_{S^c}(\beta)$.  Grouping the terms in
\eqref{eq:adj} and \eqref{eq:zonotope-area} by the cut gives
\eqref{eq:adjB} and \eqref{eq:surface-area}, including the individual
facet formula.  Finally $\mb n_b\cdot\mb w_{ij}$ is zero for an
internal edge and is $\pm1$ for an edge crossing the cut, proving
\eqref{eq:rb}.
\end{proof}

The volume, adjugate, and total surface-area formulas remain valid
for nonnegative weights whenever the zonotope is full dimensional;
a cut whose coefficient is zero need no longer index an actual facet.
The entries of $B_\beta$ and the polynomials $d,\gamma_b,r_b$ are
homogeneous in $\beta$ of degrees $1,n,n-1,1$, respectively.

\subsection{Canonical affine position}

The zone matrix selects the determinant-one map and canonical representative
\begin{equation}\label{eq:canonical}
 \M_\beta:=d^{1/(2n)}B_\beta^{-1/2},
 \qquad \mathring{\Z}_\beta:=\Z(\beta,\M_\beta).
\end{equation}
For its facet pair $b$, let $\widehat{\mb n}_b$ be the unit normal
and $\widehat\gamma_b$ the area of one facet.  By
Proposition~\ref{prop:tree-cut},
\begin{equation}\label{eq:canonical-data}
 \widehat{\mb n}_b=\frac{B_\beta^{1/2}\mb n_b}{\sqrt{r_b}},
 \qquad \widehat\gamma_b=d^{-1/(2n)}\gamma_b\sqrt{r_b}.
\end{equation}
The associated surface-area tensor is
\begin{equation}\label{eq:Hbeta}
 H_\beta:=\sum_{b\in\F}\widehat\gamma_b
                    \widehat{\mb n}_b\widehat{\mb n}_b^T.
\end{equation}

\begin{proposition}[affine reduction]\label{prop:affine}
For every $\beta>0$ and $M\in\SL(n,\R)$,
\begin{equation}\label{eq:surface-detH}
 \area(\partial\Z(\beta,M))\ge2n(\det H_\beta)^{1/n}.
\end{equation}
\end{proposition}

\begin{proof}
Put $A:=M\M_\beta^{-1}$ and $R:=(A^{-1}A^{-T})^{1/2}$.
Then $R$ is symmetric positive definite and $\det R=1$.
For every vector $\mb v$,
\[
 |A^{-T}\mb v|^2
 =\mb v^TA^{-1}A^{-T}\mb v=|R\mb v|^2.
\]
The determinant-one map $A$ carries the canonical representative
$\mathring{\Z}_\beta$ to $\Z(\beta,M)$.  Its action on facet areas,
using \eqref{eq:canonical-data}, therefore gives
\[
 \frac12\area(\partial\Z(\beta,M))
 =\sum_{b\in\F}\widehat\gamma_b|A^{-T}\widehat{\mb n}_b|
 =\sum_{b\in\F}\widehat\gamma_b|R\widehat{\mb n}_b|.
\]
For a unit vector $\mb v$, Cauchy--Schwarz gives
$|R\mb v|=|\mb v||R\mb v|\ge\mb v^TR\mb v$, with equality exactly
when $\mb v$ is an eigenvector of $R$.  The tensor $H_\beta$ is
positive definite because its positively weighted normals span
$\R^n$.  We may therefore apply trace AM--GM to the positive
definite matrix $R^{1/2}H_\beta R^{1/2}$:
\begin{equation}\label{eq:trace-am-gm}
 \begin{aligned}
 \frac12\area(\partial\Z(\beta,M))
 &\ge\sum_{b\in\F}\widehat\gamma_b
                   \widehat{\mb n}_b^TR\widehat{\mb n}_b
 =\tr(RH_\beta)=\tr(R^{1/2}H_\beta R^{1/2})\\
 &\ge n\det(R^{1/2}H_\beta R^{1/2})^{1/n}
 =n(\det H_\beta)^{1/n}.
 \end{aligned}
\end{equation}
This proves \eqref{eq:surface-detH}.  Equality in trace AM--GM holds
if and only if $R^{1/2}H_\beta R^{1/2}$ is scalar.  Equality in the
first estimate additionally requires every canonical facet normal
to be an eigenvector of $R$.
\end{proof}

\section{Weighted averaging and equality}
\label{sec:averaging}

The normal lengths supply the calibration for a single application
of Jensen's inequality.  Define two positive polynomial scalars by
\begin{align}\label{eq:phi-psi}
 \varphi:=\det\left(\sum_{b\in\F}
              \frac{\gamma_b}{|\mb n_b|}\mb n_b\mb n_b^T\right),
 \qquad
 \psi:=d^{2n-3}\det\left(\sum_{b\in\F}
              \frac{\gamma_b r_b}{|\mb n_b|^3}\mb n_b\mb n_b^T\right).
\end{align}
The polynomials are homogeneous in $\beta$ of degrees
$n(n-1)$ and $3n(n-1)$, respectively.

\begin{theorem}[graphical zonotope inequality]\label{thm:graphical}
For every positive zone vector $\beta$ and every $M\in\SL(n,\R)$,
\begin{equation}\label{eq:graphical-bound}
 \iota_n(\Z(\beta,M))
 \ge 2n\left(\frac{\varphi(\beta)^3}{\psi(\beta)}\right)^{1/(2n)}.
\end{equation}
Equality holds if and only if all zone weights are equal and
$M\in\SO(n)$.
\end{theorem}

\begin{proof}
Let $\B$ be the $n$-element subsets of $\F$ whose normals are
independent.  For $I\in\B$, set
\begin{equation}\label{eq:calibrated-basis-data}
 p_I:=\det(\mb n_b:b\in I)^2
                 \prod_{b\in I}\frac{\gamma_b}{|\mb n_b|},
 \qquad
 \hat r_b:=\frac{r_b}{|\mb n_b|^2},
 \qquad
 \hat r_I:=\prod_{b\in I}\hat r_b.
\end{equation}
By Cauchy--Binet,
\begin{equation}\label{eq:scalar-basis-sums}
 \varphi=\sum_{I\in\B}p_I,
 \qquad
 \psi=d^{2n-3}\sum_{I\in\B}p_I\hat r_I.
\end{equation}
Consequently $\hat p_I:=p_I/\varphi$ are positive weights of sum
one, and
$\sum_I\hat p_I\hat r_I=\psi/(d^{2n-3}\varphi)$.

Put $T:=\sum_b\gamma_b r_b^{-1/2}\mb n_b\mb n_b^T$.
Equations \eqref{eq:canonical-data}--\eqref{eq:Hbeta} give
\[
 H_\beta=d^{-1/(2n)}B_\beta^{1/2}T B_\beta^{1/2},
 \qquad \det H_\beta=d^{1/2}\det T.
\]
Since $\gamma_b/\sqrt{r_b}
=(\gamma_b/|\mb n_b|)\hat r_b^{-1/2}$, Cauchy--Binet expands
$\det T$ as $\sum_{I\in\B}p_I\hat r_I^{-1/2}$.
Jensen's inequality for $x\mapsto x^{-1/2}$ therefore yields
\begin{align}
 \det H_\beta
 &=d^{1/2}\varphi
          \sum_{I\in\B}\hat p_I\hat r_I^{-1/2}\notag\\
 &\ge d^{1/2}\varphi
          \left(\sum_{I\in\B}\hat p_I\hat r_I\right)^{-1/2}
 =d^{n-1}\frac{\varphi^{3/2}}{\psi^{1/2}}.
                                                    \label{eq:detH-ac}
\end{align}
Together with \eqref{eq:volume} and \eqref{eq:surface-detH}, this proves
\eqref{eq:graphical-bound}.

Suppose equality holds.  In a series of small implications, we work
backwards from the equality in Jensen's inequality and equality in the
affine estimate to the equality of the zone parameters $\beta_e$
and the orthogonality of $M$.

\begin{enumerate} 
\item \emph{The constants $\hat r_I$ are all equal:} 
Each $p_I$ is positive, for $I\in\B$,
because it is a product of positive factors.
Equality in Jensen implies that the basis products $\hat r_I$ are all equal.
\item \emph{The constants $\hat r_b$ are all equal:}
First consider the singleton cuts $\br i$, with normals $\mb n_i$.
Every $n$ of these $n+1$ normals form a basis.  Call them \emph{singleton bases}.
Comparing the products
for the bases omitting $i$ and omitting $j$ gives
\(
\hat r_{\br i}=\hat r_{\br j}.
\)
Let $a>0$ be their common value. 
For any cut $b=\br{S\mid S^c}$, choose $i\in S$ and $j\in S^c$.
Using the choice $\mb n_b=\sum_{k\in S}\mb n_k$, we have a
basis $I=\{b\}\cup\{\br{k}\mid k\not\in\{i,j\}\}$:
\[
|\det(\mb n_b,\mb n_k:k\not\in\{i,j\})|=|\det(\mb n_k:k\ne j)|>0.
\]
Comparing this basis with a singleton basis $J$:
$\hat r_I=a^{n-1}\hat r_b = a^n=\hat r_J$.  
This gives $\hat r_b=a$ for all $b$.
\item \emph{$B_\beta$ is a scalar multiple of $I_n$:}
Set $D:=B_\beta-aI_n$.  The identity $\hat r_b=a$ says that
$\mb n_b^TD\mb n_b=0$ for every $b\in\F$.  
Polarization gives
\[
 2\mb n_i^TD\mb n_j
 =(\mb n_i+\mb n_j)^TD(\mb n_i+\mb n_j)
      -\mb n_i^TD\mb n_i-\mb n_j^TD\mb n_j=0.
\]
Thus all pairings of $D$ on a basis vanish, and $B_\beta=aI_n$.
\item \emph{The zone weights $\beta_{ij}$ are all equal:}
The individual zone weights can be recovered from
\begin{equation}\label{eq:beta-recovery}
 \mb n_i^TB_\beta\mb n_j=-\beta_{ij}\qquad(i\ne j).
\end{equation}
Indeed, in the defining sum for $B_\beta$, only edge $ij$ has a
nonzero pairing with both $\mb n_i$ and $\mb n_j$, and its two
pairings have opposite signs by \eqref{eq:zone-vector}.
Equation \eqref{eq:normal-gram} now gives
$\beta_{ij}=a(n+1)^{-(n-1)/n}$ for every edge, so all zone weights
are equal.
\item \emph{$H_\beta$ is scalar:}
Every permutation $\sigma$ of the regular frame is realized by an
orthogonal map $Q_\sigma$: the Gram matrix is unchanged, and the
frame spans $\R^n$ with its sole relation $\sum_i\mb n_i=0$.
This map permutes the cut-normal lines and preserves their facet
areas, since the weights are equal and the tree formulas are
invariant under vertex relabelling.  Therefore
$Q_\sigma H_\beta Q_\sigma^T=H_\beta$ for every $\sigma$.
It follows that $\mb n_i^TH_\beta\mb n_i$ has one common value
$\alpha$, and $\mb n_i^TH_\beta\mb n_j$ for $i\ne j$ has one common
value $\eta$.  Pairing with $\sum_j\mb n_j=0$ gives
$\alpha+n\eta=0$.  These pairings are consequently a common positive
multiple of the frame Gram matrix \eqref{eq:normal-gram}.
Because the frame spans $\R^n$, this proves $H_\beta=hI_n$ for some
$h>0$.
\item $M\in\SO(n)$:
Equality in \eqref{eq:trace-am-gm} now requires
$R^{1/2}H_\beta R^{1/2}=hR$ to be scalar.  Since $\det R=1$,
$R=I_n$.  Also $\M_\beta=I_n$, so $A=M$ and
$M^{-1}M^{-T}=R^2=I_n$.  Thus $M$ is orthogonal, and its determinant
one gives $M\in\SO(n)$.  
\end{enumerate}
Conversely, the regular-frame identities give
\[
 \sum_{i<j}\mathbf w_{ij}\mathbf w_{ij}^T
 =(n+1)^{(n-1)/n}I_n.
\]
Equal zone weights therefore make \(B_\beta\) scalar, so \(\mathring
M_\beta=I_n\) and all \(\hat r_b\) are equal. The symmetry argument
above makes \(H_\beta\) scalar. If \(M\in\mathrm{SO}(n)\), then
\(R=I_n\), and equality holds in Jensen's inequality and in both
affine estimates.
\end{proof}

\section{Three-dimensional parallelohedra}
\label{sec:three-dimensional}

We now set $n=3$.  The general results handle the truncated-octahedral
and dodecahedral families.  Only the scalar inequality in
Appendix~\ref{sec:beta-certificate} and the two prism cases remain.

\subsection{The six-zone model}

For the normalized regular frame we may take
\[
 \begin{aligned}
 \mb n_1&=2^{-2/3}(1,1,1),&
 \mb n_2&=2^{-2/3}(1,-1,-1),\\
 \mb n_3&=2^{-2/3}(-1,1,-1),&
 \mb n_4&=2^{-2/3}(-1,-1,1).
 \end{aligned}
\]
Its sum is zero and $\det(\mb n_1,\mb n_2,\mb n_3)=1$.
The reciprocal generators in \eqref{eq:zone-vector} satisfy
\[
 \mb w_{ij}=\pm(\mb n_k\times\mb n_\ell),
 \qquad \{i,j,k,\ell\}=[4].
\]
Thus \eqref{eq:param} is the standard six-zone parametrization, with
complementary edge labels \cite[\S2.1]{Langi2022}; compare the
zone-rescaling lemma \cite[Lemma~1]{McMullen1975}.  Up to translation,
every parallelohedron of truncated-octahedral type has the form
$\Z(\beta,M)$; this is essentially \cite[\S2.1]{Langi2022}, except
that L\'angi starts with any four nonzero vectors $\mb v_i$ spanning
$\R^3$ and satisfying $\sum_i\mb v_i=0$.  Choose an invertible matrix
$A$ with $\mb v_i=A\mb n_i$.  Then
\[
 \mb v_i\times\mb v_j=(\adj A)^T(\mb n_i\times\mb n_j).
\]
Since $\det((\adj A)^T)=(\det A)^2>0$, its positive scalar determinant
factor can be absorbed into the zone weights, leaving the matrix in
$\mathrm{SL}(3,\R)$.
The rhombic and elongated cases are obtained by deleting zones.
Figure~\ref{fig:1} illustrates the six parallel classes.

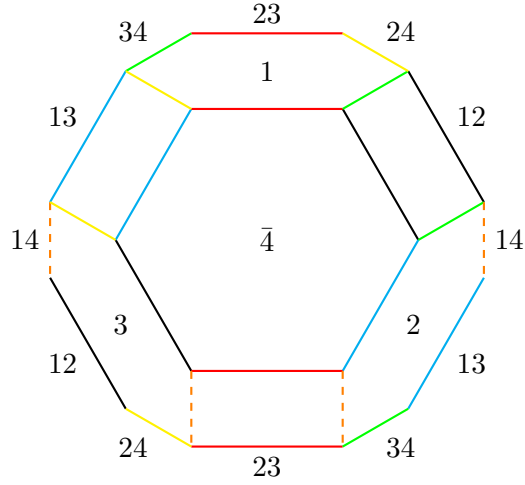
\begin{figure}[ht]
\centering
\begin{tikzpicture}[scale=2]
  \def\off{0.5}                 
  \pgfmathsetmacro{\rt}{sqrt(3)/2}

  \coordinate (A) at (1,0);
  \coordinate (B) at (0.5,\rt);
  \coordinate (C) at (-0.5,\rt);
  \coordinate (D) at (-1,0);
  \coordinate (E) at (-0.5,-\rt);
  \coordinate (F) at (0.5,-\rt);
  \coordinate (A1) at ($(A)+(\off*\rt,\off*0.5)$);
  \coordinate (B1) at ($(B)+(\off*\rt,\off*0.5)$);
  \coordinate (B2) at ($(B)+(0,\off)$);
  \coordinate (C2) at ($(C)+(0,\off)$);
  \coordinate (C1) at ($(C)+(-\off*\rt,\off*0.5)$);
  \coordinate (D1) at ($(D)+(-\off*\rt,\off*0.5)$);
  \coordinate (D2) at ($(D1)+(0,-\off)$);
  \coordinate (E2) at ($(D2)+(0.5,-\rt)$);
  \coordinate (E1) at ($(E)+(0,-\off)$);
  \coordinate (F1) at ($(F)+(0,-\off)$);
  \coordinate (F2) at ($(F1)+(\off*\rt,\off*0.5)$);
  \coordinate (A2) at ($(F2)+(0.5,\rt)$);


  \draw[thick,black] (A1) --  node[midway,anchor=south west,black]{12} (B1);
  \draw[thick,black] (A) -- (B);
  \draw[thick,black] (E) -- (D);
   \draw[thick,black] (E2) -- node[midway,anchor=north east,black]{12} (D2);
   
    \draw[thick,cyan] (A2) -- node[midway,anchor=north west,black]{13} (F2);
  \draw[thick,cyan] (A) -- (F);
  \draw[thick,cyan] (C) -- (D);
   \draw[thick,cyan] (C1) -- node[midway,anchor=south east,black]{13} (D1);

      \draw[thick,red] (C2) -- node[midway,anchor=south,black]{23} (B2);
  \draw[thick,red] (C) -- (B);
  \draw[thick,red] (E) -- (F);
   \draw[thick,red] (E1) -- node[midway,anchor=north,black]{23} (F1);
   
  \draw[thick,green] (C1) -- node[midway,anchor=south east,black]{34} (C2);
  \draw[thick,green] (B) -- (B1);
  \draw[thick,green] (A) -- (A1);
   \draw[thick,green] (F1) -- node[midway,anchor=north west,black]{34} (F2);
   
  \draw[thick,orange,dashed] (A1) -- node[midway,anchor=west,black]{14} (A2);
  \draw[thick,orange,dashed] (F) -- (F1);
  \draw[thick,orange,dashed] (E) -- (E1);
   \draw[thick,orange,dashed] (D1) -- node[midway,anchor=east,black]{14} (D2);
   
  \draw[thick,yellow] (B1) -- node[midway,anchor=south west,black] {24} (B2);
  \draw[thick,yellow] (C) -- (C1);
  \draw[thick,yellow] (D) -- (D1);
   \draw[thick,yellow] (E1) -- node[midway,anchor=north east,black] {24} (E2);
   \node at ($(B)!0.5!(C2)$) {$1$};
   \node at ($(A)!0.5!(F2)$) {$2$};
   \node at ($(D)!0.5!(E2)$){$3$};
   \node at (0,0) {$\bar4$};
\end{tikzpicture}
\caption{The six zones of a truncated octahedron, labelled by their
 directions $\R M\mb w_{ij}$.  Each zone contains six edges; hidden
 edges are omitted.  In the symmetric model, the outward normals to
 facets $1,2,3,4$ form a regular tetrahedral frame.  Facet $4$ is
 opposite~$\bar4$.}
\label{fig:1}
\end{figure}

The seven facet pairs consist of the four singleton cuts
$\F_{\rmh}=\{\br1,\br2,\br3,\br4\}$ and the three two--two cuts
$\F_{\rms}=\{\br{12\mid34},\br{13\mid24},\br{14\mid23}\}$.  The
subscripts \(\mathrm h\) and \(\mathrm s\) refer to the hexagonal and
square facets of the symmetric reference polyhedron; the square-type
facets of a general member are parallelograms.  For
$\{i,j,k,\ell\}=[4]$, the tree formulas become
\begin{equation}\label{eq:three-gamma}
 \gamma_{\br i}
 =\beta_{jk}\beta_{j\ell}+\beta_{jk}\beta_{k\ell}
                         +\beta_{j\ell}\beta_{k\ell},
 \qquad
 \gamma_{\br{ij\mid k\ell}}=\beta_{ij}\beta_{k\ell}.
\end{equation}
For $\beta=\mb1$ there are $2^4$ spanning trees, so $d=16$, and
\begin{equation}\label{eq:Arch}
 \gamma_b=\begin{cases}3,&b\in\F_{\rmh},\\1,&b\in\F_{\rms},\end{cases}
 \qquad
 |\mb n_b|=\begin{cases}2^{-2/3}\sqrt3,&b\in\F_{\rmh},\\
                         2^{-2/3}2,&b\in\F_{\rms}.
             \end{cases}
\end{equation}
Consequently
\[
 \area(\partial\Z_{\mathrm p,3})
 =12\,(1+2\sqrt3)2^{-2/3},
 \qquad
 \iota_3(\Z_{\mathrm p,3})=\frac{3(1+2\sqrt3)}{4^{2/3}}.
\]
This gives
\begin{equation}\label{eq:kappa}
 \iota_3(\Z_{\mathrm p,3})=6\kappa^{1/3},\qquad  \kappa:=\kappa_3=\frac{(1+2\sqrt3)^3}{128}.
\end{equation}

\subsection{The truncated-octahedral case}

The remaining scalar estimate, proved in
Appendix~\ref{sec:beta-certificate}, is
\begin{equation}\label{eq:direct-zone-weight}
 \varphi(\beta)^3\ge\kappa^2\psi(\beta)
 \qquad(\beta\in\R_{\ge0}^{\E},\ n=3).
\end{equation}

\begin{theorem}[truncated-octahedral case]\label{thm:main}
For every positive zone vector $\beta$ and every $M\in\SL(3,\R)$,
\[
 \iota_3(\Z(\beta,M))\ge\frac{3(1+2\sqrt3)}{4^{2/3}}.
\]
Equality holds if and only if all six weights are equal and
$M\in\SO(3)$.
\end{theorem}

\begin{proof}
The affine reduction, Jensen's inequality, and
\eqref{eq:direct-zone-weight} give
\begin{equation}\label{eq:chain}
 \area(\partial\Z(\beta,M))
 \ge6(\det H_\beta)^{1/3}
 \ge6d^{2/3}\left(\frac{\varphi^3}{\psi}\right)^{1/6}
 \ge6\kappa^{1/3}d^{2/3}.
\end{equation}
Since $d=\vol_3(\Z(\beta,M))$, this is the asserted bound.
Equality requires equality in Theorem~\ref{thm:graphical}, which
forces equal weights and orthogonal $M$.  Conversely those conditions
give $\Z_{\mathrm p,3}$ up to similarity and equality.
\end{proof}

\subsection{The two dodecahedral cases}

\begin{corollary}\label{cor:dodecahedra}
Every rhombic-dodecahedral parallelohedron satisfies
\[
 \iota_3(\Z)\ge3\cdot2^{5/6},
\]
with equality exactly for Kepler's regular rhombic dodecahedron.
Every genuine elongated dodecahedron satisfies the strict inequality
$\iota_3(\Z)>3\cdot2^{5/6}$; the constant is sharp on its closure.
\end{corollary}

\begin{proof}
In the six-zone model, these classes have, after relabelling,
\[
 \beta_{34}=0,\qquad \beta_{12}\ge0,\qquad
 \beta_{13},\beta_{14},\beta_{23},\beta_{24}>0,
\]
with $\beta_{12}=0$ for the rhombic class and $\beta_{12}>0$ for the
elongated class \cite[\S2.1]{Langi2022}.
Orient the four core directions around a cycle:
\[
 (\mb u_1,\mb u_2,\mb u_3,\mb u_4)
 =(M\mb w_{13},-M\mb w_{23},M\mb w_{24},-M\mb w_{14}).
\]
They sum to zero, and every triple has determinant of absolute value
one, since it is a spanning tree of $\K_4$.
The remaining direction is
$M\mb w_{12}=\mb u_1+\mb u_2$.
Thus Theorems~\ref{thm:rhombic} and \ref{thm:elongated} apply with
$S=\{1,2\}$, $t=\beta_{12}$, and the corresponding four core weights.
Their equality and sharpness statements give the result, since
$2c_3=3\cdot2^{5/6}$.
\end{proof}

\subsection{Prisms}
\label{subsec:prism-cases}

The two remaining Fedorov classes are prisms over centrally symmetric
polygons.  A parallelepiped is a prism over a parallelogram.  The
hexagonal-prism type has three generating segments in a common plane
and one transverse segment.  The zonotope is therefore the Minkowski
sum of a centrally symmetric planar hexagon and an interval, hence an
actual prism.  We treat both classes by the same sharp prism
inequality.

Steiner solved the prism problem~\cite{Steiner1842}.  P\'olya treats
the right-prism optimization and the comparison with oblique
prisms~\cite[pp.~138--140]{polya1990mathematics}; see also Fejes
T\'oth's discussion of Steiner's optimality
condition~\cite[p.~134]{toth2023lagerungen}.  The isoperimetric
inequality for polygons goes back to Zenodorus; see also
\cite{FejesToth1964}.

\begin{lemma}[sharp inequality for polygonal prisms]
\label{lem:polygonal-prism}
Let $\Z_0$ be a convex $m$-gon in a plane in $\R^3$, and let
$\Z=\Z_0+[-\mb u/2,\mb u/2]$ be a nondegenerate, possibly oblique
prism.  Then
\begin{equation}\label{eq:sharp-polygonal-prism}
 \iota_3(\Z)\ge6\left(\frac{m\tan(\pi/m)}4\right)^{1/3}.
\end{equation}
Equality holds if and only if the base is regular, the prism is
right, and its height is twice the inradius of the base.
\end{lemma}

\begin{proof}
Let $a$ and $p$ be the area and perimeter of the base $\Z_0$, and
let $h$ be the distance between the two base planes.  Each lateral
parallelogram has area at least $h$ times the length of its base edge.
Consequently
\[
 \area(\partial\Z)\ge2a+hp,\qquad \vol_3(\Z)=ah.
\]
Equality for an individual lateral face requires the component of
$\mb u$ parallel to the base plane to be parallel to that base edge.
Equality for all lateral faces therefore requires this component to
vanish, since a nondegenerate polygon has nonparallel edges.  Thus
equality in the area estimate holds exactly when the prism is right.

Put $x=hp/a$ and $a_m:=m\tan(\pi/m)$, the area of a regular
$m$-gon with inradius one.  The polygonal isoperimetric inequality
gives $p^2/a\ge4a_m$, with equality exactly for a regular $m$-gon.
Hence
\[
 \iota_3(\Z)
 \ge(2+x)x^{-2/3}\left(\frac{p^2}{a}\right)^{1/3}
 \ge(2+x)x^{-2/3}(4a_m)^{1/3}.
\]
The function of $x>0$ in this expression has derivative
\[
 \frac{d}{dx}\bigl((2+x)x^{-2/3}\bigr)
 =\frac{x-4}{3x^{5/3}},
\]
so its unique minimum is attained at $x=4$.  Substitution gives
$\iota_3(\Z)\ge6(a_m/4)^{1/3}$.
Equality requires a right prism, a regular base, and $hp=4a$.
For a regular polygon with inradius $r$, $a=rp/2$, so the last
condition is $h=2r$.  These conditions are also sufficient.
\end{proof}

\begin{proposition}[two prism classes]
\label{prop:parallelepipeds-and-hexagonal-prisms}
The following sharp inequalities hold.
\begin{enumerate}[label=\textup{(\roman*)}]
\item Every parallelepiped $\Z$ satisfies
\begin{equation}\label{eq:parallelepiped-bound}
 \iota_3(\Z)\ge6.
\end{equation}
Equality holds if and only if $\Z$ is a cube.

\item Every parallelohedron $\Z$ of hexagonal-prism type satisfies
\begin{equation}\label{eq:hexagonal-prism-bound}
 \iota_3(\Z)\ge3\cdot2^{2/3}3^{1/6}.
\end{equation}
Equality holds if and only if $\Z$ is a right regular hexagonal prism
whose height is twice the inradius of its base.
\end{enumerate}
\end{proposition}

\begin{proof}
Take $m=4$ and $m=6$ in Lemma~\ref{lem:polygonal-prism}.
For $m=4$, the regular base is a square and twice its inradius is
its side length, so the equality prism is a cube.
For $m=6$, the sharp constant simplifies to
$3\cdot2^{2/3}3^{1/6}$, and the equality case is the stated right
regular hexagonal prism.
\end{proof}

\subsection{Completion of the proof}
\label{sec:completion}

\begin{proof}[Proof of Theorem~\ref{thm:global-main}]
The quotient is invariant under translations, rotations, and positive
homotheties.  Fedorov's classification and the preceding results give
\begin{center}
\begin{tabular}{@{}lll@{}}
\toprule
Fedorov type & bound for $\iota_3(\Z)$ & reference\\
\midrule
truncated octahedron & $\ge3(1+2\sqrt3)/4^{2/3}$ &
 Theorem~\ref{thm:main}\\
elongated dodecahedron & $>3\cdot2^{5/6}$ &
 Corollary~\ref{cor:dodecahedra}\\
rhombic dodecahedron & $\ge3\cdot2^{5/6}$ &
 Corollary~\ref{cor:dodecahedra}\\
hexagonal prism & $\ge3\cdot2^{2/3}3^{1/6}$ &
 Proposition~\ref{prop:parallelepipeds-and-hexagonal-prisms}\\
parallelepiped & $\ge6$ &
 Proposition~\ref{prop:parallelepipeds-and-hexagonal-prisms}\\
\bottomrule
\end{tabular}
\end{center}
The four non-truncated types have constants strictly greater than
$\iota_3(\Z_{\mathrm p,3})$.  In the truncated-octahedral type,
Theorem~\ref{thm:main} gives equality precisely for the Archimedean
similarity class.
\end{proof}

\appendix
\section{Exact Bernstein certificate in dimension three}
\label{sec:beta-certificate}

Throughout this appendix $n=3$.  We prove \eqref{eq:direct-zone-weight}
by an exact tensor-product Bernstein certificate for
\begin{equation}\label{eq:app-f-beta}
 f(\beta):=f_3(\beta)=\varphi(\beta)^3-\kappa^2\psi(\beta).
\end{equation}
All certificate arithmetic lies in $\Q(\sqrt3)$.  Only nonnegativity
is required: the equality case of the geometric theorem was already
settled analytically in Theorem~\ref{thm:graphical}.

\subsection{Reconstructing the polynomial}

The seven normal lines in Section~\ref{sec:three-dimensional} realize
the rank-three non-Fano matroid \cite{Oxley}.  Their six dependent
triples are
\[
 \{\br i,\br j,\br{ij\mid k\ell}\},
 \qquad \{i,j,k,\ell\}=[4],\quad ij\in\E.
\]
The other $29$ triples form the set $\B$ of normal bases.
Of these, $28$ have squared determinant $1$, while the triple
$\F_{\rms}$ has squared determinant $4$.  These values follow by
direct computation from the normalized frame.

The graph formulas give $d=\tau(\beta)$, the coefficients
\eqref{eq:three-gamma}, and the cut weights \eqref{eq:rb}.
For exact reconstruction, put
\[
 \varepsilon_I:=\prod_{b\in I}|\mb n_b|
       =2(\sqrt3/2)^{|I\cap\F_{\rmh}|}\qquad(I\in\B).
\]
The calibrated data \eqref{eq:calibrated-basis-data} can then be
computed directly in $\Q(\sqrt3)$ as
\begin{equation}\label{eq:app-source-data}
 p_I=\frac{\det(\mb n_b:b\in I)^2}{\varepsilon_I}
                         \prod_{b\in I}\gamma_b,
 \qquad
 \hat r_I=\frac{\prod_{b\in I}r_b}{\varepsilon_I^2}.
\end{equation}
Thus \eqref{eq:scalar-basis-sums} specializes to
\begin{equation}\label{eq:app-scalars}
 \varphi=\sum_{I\in\B}p_I,
 \qquad \psi=d^3\sum_{I\in\B}p_I\hat r_I.
\end{equation}
The degrees of $p_I$ and $\hat r_I$ are $6$ and $3$;
those of $\varphi$ and $\psi$ are $6$ and $18$.
In particular $f\in\Q(\sqrt3)[\beta]$ is homogeneous of degree $18$.
No expanded polynomial is needed as input to the verifier.

\begin{theorem}[direct zone-weight certificate]\label{thm:beta-certificate}
For every nonnegative zone vector $\beta\in\R_{\ge0}^{\E}$,
\begin{equation}\label{eq:app-f-nonnegative}
 f(\beta)\ge0.
\end{equation}
\end{theorem}

\subsection{Normalization}

The construction is invariant under the $S_4$-action relabelling the
four vertices.  For $\beta\ne0$, homogeneity and edge transitivity
allow us to normalize
\begin{equation}\label{eq:app-normalization}
 \beta_{34}=1,
 \qquad 0\le\beta_{12},\beta_{13},\beta_{14},\beta_{23},\beta_{24}\le1.
\end{equation}
Set
\begin{equation}\label{eq:app-g-definition}
 g(\beta_{12},\beta_{13},\beta_{14},\beta_{23},\beta_{24})
 :=f(\beta_{12},\beta_{13},\beta_{14},\beta_{23},\beta_{24},1).
\end{equation}
The exact expansion of $f$ has $6{,}083$ monomials and multidegree
$(6,6,6,6,6,6)$.  The normalized polynomial $g$ also has $6{,}083$
monomials, minimum total degree $12$, maximum total degree $18$, and
multidegree $(6,6,6,6,6)$.  The verifier reconstructs these data from
\eqref{eq:app-source-data}--\eqref{eq:app-scalars}.

\subsection{The finite Bernstein cover}

After an affine change from a box to the unit cube, a polynomial of
coordinate degrees $m_1,\ldots,m_5$ has a tensor-product Bernstein
expansion
\[
 \sum_{0\le\alpha_i\le m_i}b_\alpha
       \prod_{i=1}^5\binom{m_i}{\alpha_i}
                         z_i^{\alpha_i}(1-z_i)^{m_i-\alpha_i}.
\]
All the basis functions are nonnegative on the cube.  Nonnegative
coefficients therefore prove nonnegativity on the box; see
\cite{titi2019matrix}.

Put $a=1/4$, and split each normalized coordinate into $[0,a]$ and
$[a,1]$.  Of the $32$ macro-boxes, four are treated separately:
\begin{equation}\label{eq:app-exceptional-boxes}
 \text{all large},\quad \text{all small},\quad
 \{\beta_{13},\beta_{24}\}\text{ small},\quad
 \{\beta_{14},\beta_{23}\}\text{ small}.
\end{equation}
The other $28$ macro-boxes have nonnegative Bernstein coefficients
directly.  In each opposite-pair box, split the three large
coordinates at $1/2$, giving $8$ subboxes for each pair and $16$ in
all.  Their Bernstein coefficients are also nonnegative.

For the all-small box, re-index the five coordinates by
$1,\ldots,5$, choose an index $j$ of a largest coordinate, and put
\begin{equation}\label{eq:app-origin-chart}
 \beta_j=t,\quad \beta_i=tu_i\ (i\ne j),\qquad
 0\le t\le\tfrac14,\quad 0\le u_i\le1.
\end{equation}
With $u_j=1$ suppressed, the minimum degree gives the factorization
\begin{equation}\label{eq:app-origin-factor}
 g(tu)=t^{12}g_j(t,u).
\end{equation}
Each of the five polynomials $g_j$ has multidegree $(6,6,6,6,6)$
and nonnegative Bernstein coefficients on its chart box.

For the all-large box, put $z_i=1-\beta_i$, choose a largest deviation
$z_j=t$, and write $z_i=tu_i$.  The chart domain is
\begin{equation}\label{eq:app-large-chart}
 0\le t\le\tfrac34,\quad 0\le u_i\le1,\quad u_j=1.
\end{equation}
The translated polynomial $g(1-z)$ has no term of total degree less
than two.  Hence
\begin{equation}\label{eq:app-large-factor}
 g(1-tu)=t^2h_j(t,u).
\end{equation}
Each of the five $h_j$ has multidegree $(16,6,6,6,6)$ and
nonnegative Bernstein coefficients on its chart box.

The complete cover is summarized below.  For multidegree
$(m_1,\ldots,m_5)$ there are $\prod_i(m_i+1)$ coefficients.
Every listed coefficient is checked to be nonnegative.
\begin{center}
\begin{tabular}{@{}lrr@{}}
\toprule
certificate class & rectangles & coefficients per rectangle\\
\midrule
ordinary macro-boxes & $28$ & $7^5$\\
opposite-pair subboxes & $16$ & $7^5$\\
all-small charts & $5$ & $7^5$\\
all-large charts & $5$ & $17\cdot7^4$\\
\midrule
all certificates & $54$ &\\
\bottomrule
\end{tabular}
\end{center}

\subsection{Exact certificate criterion}

\begin{proposition}[finite certificate criterion]
\label{prop:app-finite-certificate}
Suppose the following checks have been made in exact arithmetic:
\begin{enumerate}[label=\textup{(\alph*)},leftmargin=2.4em]
\item the $16$ spanning trees and $29$ normal bases, with their
squared determinants, are reconstructed, and the polynomial $f$ is
formed from \eqref{eq:app-source-data}--\eqref{eq:app-scalars};
\item homogeneity, invariance under all $24$ vertex permutations,
the stated degree data, and the factorizations
\eqref{eq:app-origin-factor} and \eqref{eq:app-large-factor} are verified;
\item every Bernstein coefficient on the $54$ rectangles in the
table is nonnegative.
\end{enumerate}
Then Theorem~\ref{thm:beta-certificate} holds.
\end{proposition}

\begin{proof}
The zero zone vector satisfies $f(0)=0$.  Every other nonnegative
zone vector can be rescaled and relabelled as in
\eqref{eq:app-normalization}.  The $28$ ordinary boxes and the
$16$ opposite-pair subboxes cover all but the all-small and all-large
boxes.  The five maximal-coordinate charts cover the all-small box,
and the five maximal-deviation charts cover the all-large box.
Their polynomial bounds prove $g\ge0$ after restoring the factors
$t^{12}$ and $t^2$, which are nonnegative also at $t=0$.
The cover therefore proves $f\ge0$ everywhere.  No information about
the zero faces of the Bernstein expansions is required.
\end{proof}

\subsection{The verification package}

The package contains
\texttt{verify\_appendix.py} and \texttt{exact\_bernstein.py}.
From its directory, run
\begin{center}
\texttt{python verify\_appendix.py}.
\end{center}
A successful run ends with
\texttt{ALL EXACT BETA-BERNSTEIN CHECKS PASSED}.

The program reconstructs the graph data, the source polynomials,
and every Bernstein coefficient from the formulas above.  No
expanded coefficient table or saved status file is read as a proof
input.  All accepted signs are decided exactly in $\Q(\sqrt3)$.
NumPy is used only to store arrays of Python integers; no
floating-point arithmetic, numerical optimization, interval
arithmetic, or quantifier elimination is used.  File hashes,
dependency versions, coefficient totals, and runtime information
appear in the package's \texttt{README} file.

\begin{proof}[Proof of Theorem~\ref{thm:beta-certificate}]
The verifier checks conditions \textup{(a)}--\textup{(c)} of
Proposition~\ref{prop:app-finite-certificate}.  That proposition
proves the theorem.
\end{proof}

\section*{Acknowledgments}

\subsection*{AI interactions}

A proof of local minimality was obtained entirely by L. Song without any machine assistance and independently of the Cesaroni--Novaga preprint.
After that, the GPT-5.6 Sol model in ChatGPT was used in a substantial way at all stages, including proof discovery and manuscript drafts.
It wrote the Python code, test suites, and code documentation. 
ChatGPT also provided editorial suggestions and feedback on the manuscript.
The authors substantially reworked and simplified the proofs. 
The authors take full responsibility for the correctness of the proof, the computer code, and the supplied certificate.

\subsection*{Human interactions}

L. Song owes a quiet debt to T. Hales for his support, encouragement, and advice throughout the years. L. Song would like to thank Dima Arinkin for teaching the representation theory essential to the symmetry reduction idea. This work is based on an honors thesis of L. Song, who would also like to thank Jason DeBlois, David Kinderlehrer, and Jonathan Rubin for taking the time to be on the thesis defense committee.

\subsection*{Funding}

This work was supported by an undergraduate travel grant from the American Mathematical Society.

\bibliographystyle{alpha}\small
\bibliography{trunc_oct_conj_bibliography}

\end{document}